\documentclass[a4paper,12pt]{amsart}
\usepackage{amssymb}
\usepackage{ifthen}
\usepackage{cite}
 \usepackage[dvips]{graphicx}
\nonstopmode\numberwithin{equation}{section}
\usepackage{amssymb}
\usepackage{ifthen}
\usepackage{graphicx}
\usepackage{amsmath}
\usepackage[T1]{fontenc} 
\usepackage[utf8]{inputenc}
\usepackage[usenames,dvipsnames]{color}
\usepackage{color}
\usepackage[english]{babel}
\usepackage{fancyhdr}
\usepackage{fancybox}
\usepackage{tikz}

\nonstopmode \numberwithin{equation}{section}
\theoremstyle{plain}

\newtheorem{conj}{Conjecture}

\theoremstyle{definition}

\newtheorem{thm}{Theorem}[section]
\newtheorem{prob}{Problem}[section]
\newtheorem{cor}{Corollary}[section]
\newtheorem{ques}{Question}[section]
\newtheorem{prop}{Proposition}[section]
\newtheorem{rem}{Remark}[section]
\newtheorem{lem}{Lemma}[section]

\newcounter{minutes}
\divide\time by 60
\newcounter{hours}
\multiply\time by 60
\addtocounter{minutes}{-\time}

\newcounter {own}
\def\theown {\thesection       .\arabic{own}}

\newenvironment{pf}[1][]{%
 \vskip 3mm
 \noindent
 \ifthenelse{\equal{#1}{}}%
  {{\slshape Proof. }}%
  {{\slshape #1.} }%
 }%
{\qed\bigskip}

\newcounter{alphabet}

\def\be{\begin{equation}}
\def\ee{\end{equation}}

\newcommand{\bee}{\begin{enumerate}}
\newcommand{\eee}{\end{enumerate}}

\newcommand{\blem}{\begin{lem}}
\newcommand{\elem}{\end{lem}}
\newcommand{\bthm}{\begin{thm}}
\newcommand{\ethm}{\end{thm}}
\newcommand{\bcor}{\begin{cor}}
\newcommand{\ecor}{\end{cor}}
\newcommand{\beg}{\begin{examp}}
\newcommand{\eeg}{\end{examp}}
\newcommand{\begs}{\begin{examples}}
\newcommand{\eegs}{\end{examples}}

\newcommand{\bdefn}{\begin{defn}}
\newcommand{\edefn}{\end{defn}}

\newcommand{\bprob}{\begin{prob}}
\newcommand{\eprob}{\end{prob}}
\newcommand{\bei}{\begin{itemize}}
\newcommand{\eei}{\end{itemize}}

\newcommand{\bcon}{\begin{conj}}
\newcommand{\econ}{\end{conj}}
\newcommand{\bcons}{\begin{conjs}}
\newcommand{\econs}{\end{conjs}}
\newcommand{\bprop}{\begin{prop}}
\newcommand{\eprop}{\end{prop}}
\newcommand{\br}{\begin{rem}}
\newcommand{\er}{\end{rem}}
\newcommand{\brs}{\begin{rems}}
\newcommand{\ers}{\end{rems}}
\newcommand{\bo}{\begin{obser}}
\newcommand{\eo}{\end{obser}}
\newcommand{\bos}{\begin{obsers}}
\newcommand{\eos}{\end{obsers}}
\newcommand{\bpf}{\begin{pf}}
\newcommand{\epf}{\end{pf}}
\newcommand{\ba}{\begin{array}}
\newcommand{\ea}{\end{array}}
\newcommand{\beq}{\begin{eqnarray}}
\newcommand{\beqq}{\begin{eqnarray*}}
\newcommand{\eeq}{\end{eqnarray}}
\newcommand{\eeqq}{\end{eqnarray*}}

\begin{document}

\title{Refined Bohr inequality for functions in $\mathbb{C}^n$ and in complex Banach spaces}

\author{Sabir Ahammed}
\address{Sabir Ahammed, Department of Mathematics, Jadavpur University, Kolkata-700032, West Bengal,India.}
\email{sabira.math.rs@jadavpuruniversity.in}

\author{Molla Basir Ahamed$ ^* $}
\address{Molla Basir Ahamed, Department of Mathematics, Jadavpur University, Kolkata-700032, West Bengal,India.}
\email{mbahamed.math@jadavpuruniversity.in}

\subjclass[{AMS} Subject Classification:]{Primary 30A10; 30B10; 30C45; 32A05; 32A10; 32K05}

\def\thefootnote{}
\footnotetext{ {\tiny File:~\jobname.tex,
printed: \number\year-\number\month-\number\day,
          \thehours.\ifnum\theminutes<10{0}\fi\theminutes }
} \makeatletter\def\thefootnote{\@arabic\c@footnote}\makeatother
\begin{abstract} 
In this paper, we first obtain a refined version of the Bohr inequality of norm-type for holomorphic mappings with lacunary series on the polydisk in $\mathbb{C}^n$ under some restricted conditions. Next, we determine the refined version of the Bohr inequality for holomorphic functions defined on a balanced domain $ G $ of a complex Banach space $ X $ and take values from the unit disk $ \mathbb{D} $. Furthermore, as a consequence of one of this results, we obtain a refined version of the Bohr-type inequalities for harmonic functions $ f=h+\bar{g} $ defined on a balanced domain $ G\subset X $. All the results are proved to be sharp.  
\end{abstract}
\maketitle
\pagestyle{myheadings}
\markboth{S. Ahammed and M. B. Ahamed }{Refined Bohr inequality for functions in $\mathbb{C}^n$ and in complex Banach spaces}
\section{Introduction}
Bohr’s power series theorem was discovered a century ago in the context of the study of Bohr’s
absolute convergence problem for the Dirichlet series is now an active area of research for different function spaces. In \cite{Bohr-1914}, Harald Bohr proved that for every holomorphic function $ f $ on the unit disc $ \mathbb{D}:=\{z\in\mathbb{C}: |z|<1\} $
\begin{align}\label{BS-eq-1.1}
\sup_{|z|\leq\frac{1}{3}}\sum_{n=0}^{\infty}\bigg|\frac{f^{(n)}(0)}{n!}z^n\bigg|\leq ||f||_{\infty}:=\sup_{z\in \mathbb{D}}|f(z)|,
\end{align}
and the radius $ 1/3 $ is optimal. The constant $ 1/3 $ is famously known as the Bohr radius and \eqref{BS-eq-1.1} is known as the Bohr inequality for the class $ \mathcal{B} $ of analytic self-maps on unit disk $ U $. \vspace{1.2mm}

Several investigations and new problems on Bohr’s inequality
in the one complex variable appeared in the literature  (see \cite{Abu-CVEE-2010,Arora-CVEE-2022,Alkhaleefah-Kayumov-Ponnusamy-PAMS-2019,Allu-Arora-JMAA-2022,Kayumov-Pon-CMFT-2017,Kayumov-Ponnusamy-JMAA-2018,Kayumov-Khammatova-Ponnusamy-JMAA-2021} and references therein). However, a detailed account of research on the Bohr radius problem, what is known as Bohr's phenomenon, can be found in the survey article \cite{Ponnusmy-Survey}. Also, references on the problem
of Bohr’s phenomenon can be found in the research book \cite{Kresin-Book-1903}. Actually, Bohr’s theorem received greater interest in $ 1995 $ after it was used by Dixon \cite{Dixon & BLMS & 1995} to characterize Banach algebras that satisfy von Neumann's inequality. Since then, a lot of research
has been devoted to extend Bohr’s result in multidimensional and abstract settings. In fact, the generalization of Bohr’s theorem for various function spaces is now an active area of research and different versions of the Bohr inequality are established in the last three decades. For instance, Aizenberg \textit{et al.} \cite{aizenberg-2001}, Aytuna and Djakov \cite{Aytuna-Djakov-BLMS-2013} have studied the Bohr property of bases for holomorphic functions; Ali \textit{et al.} \cite{Ali-Abdul-NG-CVEE-2016} have found the Bohr radius for the class of starlike log-harmonic mappings; while Paulsen \textit{et al.} \cite{Paulsen-PLMS-2002} have extended the Bohr inequality to Banach algebras; Hamada \emph{et al.}\cite{Hamada-IJM-2009} have studied the Bohr's theorem for holomorphic mappings with values in homogeneous balls; Galicer \emph{et al.}\cite{Galicer-Mansilla-Muro-TAMS-2020} have studied mixed Bohr radius in several complex variables, and many authors have studied Bohr-type inequalities for different class of functions (see \cite{Aha-Aha-CMFT-2023,Aha-Allu-RMJ-2022,Alkhaleefah-Kayumov-Ponnusamy-PAMS-2019,Allu-Arora-JMAA-2022,Das-JMAA-2022,Lata-Singh-PAMS-2022,S. Kumar-PAMS-2022,Kumar-JMAA-2023} and references therein.)  However, it can be noted that not every class of functions has the Bohr phenomenon, for example, B\'an\'at\'aau \emph{et al.} \cite{Beneteau-2004} showed that there is no Bohr phenomenon in the Hardy space $ H^p(\mathbb{D},X), $ where $p\in [1,\infty).$ In \cite{Liu-Liu-JMAA-2020}, Liu and Liu have shown that Bohr's inequality fails to hold for the class $ \mathcal{H}(\mathbb{D}^2, \mathbb{D}^2) $, a set of holomorphic functions $ f : \mathbb{D}^2\rightarrow \mathbb{D}^2 $ having lacunary series expansion. \vspace{1.2mm}

Studying Bohr inequality in view of its refined and improved form, and also establishing them as sharp for classes of analytic self-maps and also for certain classes of harmonic mappings on unit disk $ \mathbb{D} $ are extensively studied in recent years. For a detailed study of refined Bohr inequality, we refer to the articles \cite{Liu-Ponnusamy-PAMS-2021,Liu-Liu-Ponnusamy-2021} and references therein. The Bohr inequality has actually been little studied for a functions of several complex variables. Not only that, but whether the refined version of this inequality can be established for vector-valued functions, holomorphic functions with series representation in $ \mathbb{C}^n $, and more to the point, whether refined versions of the same in Banach spaces can be established, and whether they will be sharp, have not been studied yet. Recently,  Liu \emph{et al.}\cite{Liu-Liu-Ponnusamy-2021} established several refined versions of Bohr's inequality in the case of $ f\in\mathcal{H}(\mathbb{D}, \mathbb{D}) $. We recall the result here.
\begin{thm}\label{Thmm-2.1}
Suppose that $f(z)=\sum_{n=0}^{\infty}a_{pn+m}z^{pn+m} \in \mathcal{H}(\mathbb{D}, \mathbb{D}) ,$ where $p\in \mathbb{N}$ and $0\leq m\leq p$. Then
\begin{align*}
\sum_{n=1}^{\infty}|a_{pn+m}|r^{pn+m}+\left(\dfrac{1}{1+|a_m|}+\dfrac{r^p}{1-r^p}\right)\sum_{n=2}^{\infty}|a_{pn+m}|^2r^{{2pn+m}}\leq 1
\end{align*}
for $|z|=r\leq r_{p,m}(|a_m|),$ where  $r_{p,m}(|a_m|)$ is the unique positive root of the equation  
\begin{align*}
(1-|a_m|-|a_m|^2)r^{p+m}+r^p+|a_m|r^m-1=0.
\end{align*}
Further, we have $r_{p,m}(|a_m|)\geq \sqrt[p]{1/(2+|a_m|)}.$
\end{thm} 
Our main interest in this paper is not only to give several extensions of the Bohr inequality in terms of refined formulations for functions discussed above but also to establish their sharpness. To be more precise, the aim of this article is to address a refined version of the multidimensional analog of Theorem \ref{Thmm-2.1} for holomorphic and harmonic functions on a balanced domain $ G $ of a complex Banach space $ X $. The derived results of this paper reduce to the corresponding results in one complex variable.  
\section{Refined Bohr's inequality for vector-valued functions in several complex variables}
The study of the Bohr phenomenon also gets much attention from several researchers when Boas and Khavinson \cite{Boas-1997} extended the concept of the Bohr radius problem from bounded analytic functions on $\mathbb{D}$ to holomorphic functions in several complex variables (see, e.g. \cite{Aizenberg-Djakov-PAMS-2000,Paulsen-PLMS-2002,Hamada-IJM-2009,Galicer-Mansilla-Muro-TAMS-2020,S. Kumar-PAMS-2022,Lin-Liu-Ponnusamy-Acta-2023}). More precisely to say, in 1$ 997 $, Boas and Khavinson \cite{Boas-1997} introduced the $ N $-dimensional Bohr radius $ K_N, $ $ (N>1$ ) for the polydisk $ \mathbb{D}^N=\mathbb{D}\times\cdots\times \mathbb{D} $ which generates extensive research activity in Bohr radius problems. Actually, Boas and Khavinson \cite{Boas-1997} proved that the $ N $-Bohr radius $ K_N $ as the largest radius $ r>0 $ such that for every complex polynomials $ \sum_{\alpha\in\mathbb{N}_0^N}c_{\alpha}z^{\alpha} $ in $ N $ variables 
\begin{align*}
	\sup_{z\in r\mathbb{D}^N} \sum_{\alpha\in\mathbb{N}_0^N}|c_{\alpha}z^{\alpha}|\leq \sup_{z\in r\mathbb{D}^N} \bigg|\sum_{\alpha\in\mathbb{N}_0^N}c_{\alpha}z^{\alpha}\bigg|.
\end{align*}
As expected, the constant $ K_N $ is defined as the largest radius $ r $ satisfying $ \sum_{\alpha}|c_{\alpha}z^{\alpha}|<1 $ for all $ z $ with $ ||z||_{\infty}:=\max\{|z_1|, |z_2|, \ldots, |z_N|\}<r $ and all $ f(z)=\sum_{\alpha}c_{\alpha}z^{\alpha}\in\mathcal{H}(U^N) $. In recent years, a lot of attention has been paid to multidimensional generalizations of Bohr’s theorem. For different aspects of multidimensional Bohr phenomenon including recent advances in this topic, we refer to the articles by  Aizenberg \cite{Aizn-PAMS-2000}, Liu and Ponnusamy \cite{Liu-Ponnusamy-PAMS-2021}, Paulsen \cite{Paulsen-PLMS-2002}, Defant and Frerick \cite{Defant-Frerick-IJM-2001}, Kumar\cite{S. Kumar-PAMS-2022} and also \cite{Liu-Liu-JMAA-2020,Lin-Liu-Ponnusamy-Acta-2023} and references therein.\vspace{1.2mm}

Throughout the paper, we denote $\mathbb{N}_0$  be the set of non-negative integers, $\mathbb{D}^n$ be the open unit polydisk in $\mathbb{C}^n,$ and  $\mathcal{H}\left(X,Y\right)$ be the set of holomorphic mappings from $X$ into $Y$. Let the symbol $^\prime$ stand for transpose.  By an example, Liu and Liu \cite{Liu-Liu-JMAA-2020} have shown that the Bohr inequality of norm type for holomorphic mappings with lacunary series on the unit polydisc in $ \mathbb{C}^n $ does not hold in general. More precisely, it is shown in \cite{Liu-Liu-JMAA-2020} that for the function $ f $ given by
\begin{align*}
	f(z)=\left(z_1\frac{z_1-\frac{1}{\sqrt{2}}}{1-\frac{1}{\sqrt{2}}z_1}, z_2\frac{z_2-\frac{2}{\sqrt{5}}}{1-\frac{2}{\sqrt{5}}z_2}\right)^{\prime},\; z=(z_1, z_2)^{\prime}
\end{align*}
that $ f\in \mathcal{H}(\mathbb{D}^2, \overline{\mathbb{D}^2}) $ and 
\begin{align*}
	&\sum_{s=1}^{\infty}\frac{||D^sf(0)\left(z^s\right)||}{s!}>\frac{1}{\sqrt{2}}\left(\frac{2}{\sqrt{5}}+\frac{1}{\sqrt{2}}\right)>1\; \mbox{for}\; z=\left(\frac{1}{\sqrt{2}}, \frac{1}{\sqrt{2}}\right).
\end{align*}
To overcome the problem, the authors have studied Bohr inequality in \cite{Liu-Liu-JMAA-2020} with some restricted condition on the function $ f $ belong to the classes $ \mathcal{H}\left(\mathbb{D}^n, \mathbb{C}^n\right) $ and $ \mathcal{H}\left(\mathbb{D}^n, \overline{U^n}\right) $. In \cite{Liu-Liu-JMAA-2020}, the Bohr inequality of norm-type is investigated for holomorphic mappings $f$ for the class $\mathcal{H}\left(\mathbb{D}^n,\mathbb{C}^n\right)$ with lacunary series under some restricted conditions on the function $f.$\vspace{1.2mm}

 To establish the Bohr inequality for holomorphic functions in $ \mathbb{C}^n $, the following lemma was obtained by Liu and Liu in \cite{Liu-Liu-JMAA-2020}.
\begin{lem}\cite{Liu-Liu-JMAA-2020}\label{Lem-1.1}
Let $m\in\mathbb{N}_0,$ $N\in\mathbb{N}$ and 
\begin{align*}
\begin{cases}
\phi_1(r)= 2r^N+r-1, r\in [0,),\\
\phi_2(r)=4r^{2(N-m)}+4r^{N+1-2m}-4r^{N-2m}+r^2-2r+1, r\in[0,1), N>2m,\\
\phi_3(r)=4r^N+r^{2+2m-N}-2r^{1+2m-N}+r^{2m-N}+4r-4, r\in[0,1), m+1\leq N\leq 2m.	
\end{cases}
\end{align*}
Then there exists the maximal positive root for each $\phi_k(r)=0$ $(k=1,2,3).$
\end{lem}
In what follows, $ \lfloor x \rfloor $ denotes the largest integer not more than $ x $, where $ x $ is a real number. We recall here a key result from \cite{Liu-Liu-Ponnusamy-2021} which will be useful to prove our results of this paper.
\begin{lem}\cite{Liu-Liu-Ponnusamy-2021}\label{Lemm-1.2}
Suppose that $ f(z)=\sum_{n=0}^{\infty}a_nz^n\in \mathcal{H}(\mathbb{D}, \mathbb{D}) $. Then for any $N\in\mathbb{N}$, the following inequality holds:
\begin{align*}
\sum_{n=N}^{\infty}|a_n|r^n+sgn(t)\sum_{n=1}^{t}|a_n|^2\dfrac{r^N}{1-r}+\left(\dfrac{1}{1+|a_0|}+\dfrac{r}{1-r}\right)\sum_{n=t+1}^{\infty}|a_n|^2r^{2n}\leq (1-|a_0|^2)\dfrac{r^N}{1-r}, 
\end{align*}
\;\mbox{for}\; $|z|= r\in[0,1),$ where $t=\lfloor{(N-1)/2}\rfloor$.
\end{lem}
The Bohr inequality for functions $f\in \mathcal{H}\left(\mathbb{D}^n,\mathbb{C}^n\right)$ with lacunary series of the form
\begin{align}\label{BS-eq-2.1}
	f(z)&=(a_1z^m_1+g_1(z),a_2z_2^m+g_2(z),\dots,a_nz^m_n+g_n(z))^{\prime}\\&\nonumber= \frac{D^mf(0)\left(z^m\right)}{m!}+\sum_{s=N}^{\infty}\frac{D^{s}f(0)\left(z^{s}\right)}{s!}
\end{align}
are obtained under restricted conditions in \cite{Liu-Liu-JMAA-2020} as the following.
\begin{thm}\cite{Liu-Liu-JMAA-2020} \label{Thm-2.1}
Let $m\in\mathbb{N}_0,$  $a=(a_1,a_2,\dots, a_n)^{\prime},$ $N\geq m+1,$ $f\in \mathcal{H}\left(\mathbb{D}^n, \mathbb{C}^n\right)$ be given by \eqref{BS-eq-2.1}, $|a_l|=||a||=\max_{1\leq l\leq n}\{|a_l|\},$ $l=1,2,\dots, n$ and $a_jz^m_j+g_j(z)\in \mathcal{H}\left(\mathbb{D}^n, \overline{\mathbb{D}}\right),$ where $\frac{D^mf(0)\left(z^m\right)}{m!}=(a_1z^m_1,a_2z_2^m,\dots,a_nz^m_n)^{\prime}, $ and $j$ satisfies $|z_j|=||z||=\max_{1\leq l\leq n}\{|z_l|\}.$ Then 
\begin{align*}
\dfrac{||D^mf(0)\left(z^m\right)||}{m!}+\sum_{s=N}^{\infty}\dfrac{||D^{s}f(0)\left(z^{s}\right)||}{s!}\leq 1
\end{align*}
for $||z||=r\leq r_{N,m},$ where $r_{N,m}$ is the maximal positive root of the equation $\phi_k(r)=0$ $ (k=1, 2, 3) $ given in Lemma \ref{Lem-1.1}.
 \end{thm}
\begin{thm}\cite{Liu-Liu-JMAA-2020}\label{Thm-2.2}
Let $m\in\mathbb{N}_0,$  $N\geq m+1,$ $f(z)= \frac{D^mf(0)\left(z^m\right)}{m!}+\sum_{s=N}^{\infty}\frac{D^{s}f(0)\left(z^{s}\right)}{s!}\in \mathcal{H}\left(\mathbb{D}^n, \overline{\mathbb{D}}^n\right).$ If $\frac{|D^mf_l(0)(z^m)|}{m!}=\frac{||D^mf(0)\left(z^m\right)||}{m!},$ $l=1,2,\dots n,$ then
\begin{align*}
\frac{||D^mf(0)\left(z^m\right)||}{m!}+\sum_{s=N}^{\infty}\frac{||D^{s}f(0)\left(z^{s}\right)||}{s!}\leq 1
\end{align*}	
for $||z||=r\leq r_{N,m},$ where $r_{N,m}$ is the maximal positive root of the equation $\phi_k(r)=0$ $(k=1,2,3)$ given in Lemma \ref{Lem-1.1}.	
\end{thm}
Furthermore,  the Bohr inequality of norm type for \textit{symmetric function}s $f\in \mathcal{H}\left(\mathbb{D}^n,\overline{\mathbb{D}^n}\right)$ with lacunary series of the form 
\begin{align}\label{BS-eq-2.2}
	f(z)=\frac{D^mf(0)\left(z^m\right)}{m!}+\sum_{s=1}^{\infty}\frac{D^{sk+m}f(0)(z^{sk+m})}{(sk+m)!}
\end{align}
is established in \cite{Liu-Liu-JMAA-2020} and it is shown that the result is sharp. 
\begin{thm}\cite{Liu-Liu-JMAA-2020}\label{Thm-2.3}
Let $ m,k\in\mathbb{N}_0 $, $ 0\leq m\leq k $ and $ f\in \mathcal{H}\left(\mathbb{D}^n, \overline{\mathbb{D}^n}\right)$ be given by \eqref{BS-eq-2.2}. If $\frac{|D^mf_l(0)(z^m)|}{m!}=\frac{||D^mf(0)(z^m)||}{m!}$, l=1,2\dots, then 
\begin{align*}
&\frac{||D^mf(0)\left(z^m\right)||}{m!}+\sum_{s=1}^{\infty}\frac{||D^{sk+m}f(0)\left(z^{sk+m}\right)||}{(sk+m)!}\leq 1
\end{align*}
for $ ||z||=r\leq r_{k, m} $, where  $r_{k,m}$  is the maximal positive root of the equation 
\begin{align}\label{BS-eqqq-2.1}
	-6r^{k-m}+r^{2(k-m)}+8r^{2k}+1=0.
\end{align}
Each $ r_{k, m} $ is sharp.
\end{thm}
In recent years, refining the Bohr-type inequalities have been an active research topic. Many researchers continuously investigated refined Bohr-type inequalities for a certain class of analytic functions, for classes of harmonic mappings on the unit disk $ \mathbb{D} $, or on the shifted disk $ \Omega_{\gamma}:=\{z\in\mathbb{C} : |z+\frac{\gamma}{1-\gamma}|<\frac{1}{1-\gamma}\} $ where $ 0\leq\gamma<1 $, for operator-valued functions. For detailed information on such studies, the readers are referred to  (see\cite{Aha-Aha-CMFT-2023,Aha-All-Him-AASFM-2022,Allu-CMB-2022,Evd-Pon-Ras-RM-2021,Liu-Ponnusamy-PAMS-2021,Liu-Ponnusamy-Wang-RACSAM-2020,Liu-Liu-Ponnusamy-2021,Ponnusamy-JMAA-2022} and references therein). No one has yet explored what could be the refined version of the Bohr inequality for holomorphic functions in $ \mathbb{C}^n $. Inspired by the methods in \cite{Liu-Liu-JMAA-2020}, we are interested to investigate a refined version of the Bohr inequality for holomorphic functions in $ \mathbb{C}^n $ and also to establish their sharpness. Henceforward, the following questions arise naturally.
\begin{ques}\label{BS-qn-2.1}
Can we establish a refined version of Theorems \ref{Thm-2.1} to \ref{Thm-2.3}? Can we show them sharp keeping the radius unchanged? 
\end{ques}
To answer Question \ref{BS-qn-2.1}, we shall establish refined Bohr inequality of norm type for holomorphic mappings with lacunary series on the unit polydisk in $ \mathbb{C}^n $ under the restricted conditions that are being considered in \cite{Liu-Liu-JMAA-2020}. In fact,  we show in Theorems \ref{Thm-2.5} and \ref{Thm-2.6} that the constants $ r_{1,0}={1}/{3} $ and $ r_{2, 1}={3}/{5} $ both are optimal.  With the help of the Lemmas \ref{Lem-1.1} and \ref{Lemm-1.2}, we obtain the following result as sharp refinements of Theorem \ref{Thm-2.1}. 
\begin{thm}\label{Thm-2.5}
Let $m\in\mathbb{N}_0,$  $a=(a_1,a_2,\dots, a_n)^{\prime},$ $N\geq m+1,$ $f\in \mathcal{H}\left(U^n, \mathbb{C}^n\right)$ be given by \eqref{BS-eq-2.1}, $|a_l|=||a||=\max_{1\leq l\leq n}{|a_l|},$ $l=1,2,\dots, n$ and $a_jz^m_j+g_j(z)\in \mathcal{H}\left(\mathbb{D}^n, \overline{\mathbb{D}^n}\right),$ where $\frac{D^mf(0)\left(z^m\right)}{m!}=(a_1z^m_1,a_2z_2^m,\dots,a_nz^m_n)^{\prime}, $ and $j$ satisfies $|z_j|=||z||=\max_{1\leq l\leq n}{|z_l|}.$ Then 
\begin{align*}
\mathcal{A}^f_{m}(||z||)&:=\dfrac{||D^mf(0)\left(z^m\right)||}{m!}+\sum_{s=N}^{\infty}\dfrac{||D^{s}f(0)\left(z^{s}\right)||}{s!}+sgn(t)\sum_{s=1}^{t}\left(\dfrac{||D^{s}f(0)\left(z^{s}\right)||}{s!}\right)^2\dfrac{||z||^{N-2s}}{1-||z||}\\&\quad+\bigg(\dfrac{||z||^m}{||z||^m+\frac{||D^mf(0)\left(z^m\right)||}{m!}}+\frac{||z||}{1-||z||}\bigg)\sum_{s=t+1}^{\infty}\left(\dfrac{||D^{s}f(0)\left(z^{s}\right)||}{s!}\right)^2\leq 1
\end{align*}	
for $||z||=r\leq r_{N,m},$  where $t=\lfloor{(N-1)/2}\rfloor$ and $r_{N,m}$ is the maximal positive root of the equation $\phi_k(r)=0$ $(k=1,2,3)$ are given in Lemma \ref{Lem-1.1}.
\end{thm}
The second result we obtain is the following and it is established as a refined version of Theorem \ref{Thm-2.2}. 
\begin{thm}\label{Thm-2.6}
Let $m\in\mathbb{N}_0,$  $N\geq m+1,$ $f(z)= \frac{D^mf(0)\left(z^m\right)}{m!}+\sum_{s=N}^{\infty}\frac{D^{s}f(0)\left(z^{s}\right)}{s!}\in \mathcal{H}\left(\mathbb{D}^n, \overline{\mathbb{D}}^n\right).$ If $\frac{|D^mf_l(0)(z^m)|}{m!}=\frac{||D^mf(0)\left(z^m\right)||}{m!},$ $l=1,2,\dots n,$ then
\begin{align*}
\mathcal{B}^f_{m}(r)&:=\frac{||D^mf(0)\left(z^m\right)||}{m!}+\sum_{s=N}^{\infty}\frac{||D^{s}f(0)\left(z^{s}\right)||}{s!}+sgn(t)\sum_{s=1}^{t}\left(\frac{||D^{s}f(0)\left(z^{s}\right)||}{s!}\right)^2\dfrac{||z||^{N-2s}}{1-||z||}\\&\quad+\bigg(\frac{||z||^m}{||z||^m+\frac{||D^mf(0)\left(z^m\right)||}{m!}}+\frac{||z||}{1-||z||}\bigg)\sum_{s=t+1}^{\infty}\left(\frac{||D^{s}f(0)\left(z^{s}\right)||}{s!}\right)^2\leq 1
\end{align*}	
for $||z||=r\leq r_{N,m},$ where $t=\lfloor{(N-1)/2}\rfloor$ and  $r_{N,m}$ is the maximal positive root of the equation $\phi_k(r)=0$ $(k=1,2,3)$ given in Lemma \ref{Lem-1.1}.	
\end{thm}
The final result of this section we obtain an analogue of Theorem \ref{Thmm-2.1}  which is a refined version of Theorem \ref{Thm-2.3}.
\begin{thm}\label{BS-thm-2.7}
Let $ m,k\in\mathbb{N}_0 $, $ 0\leq m\leq k $ and $ f\in \mathcal{H}\left(\mathbb{D}^n, \overline{\mathbb{D}^n}\right) $ be given by \eqref{BS-eq-2.2}. If $\frac{|D^mf_l(0)(z^m)|}{m!}=\frac{||D^mf(0)(z^m)||}{m!}$, l=1,2\dots, then  
\begin{align*}
\mathcal{C}^f_m(||z||):=&\frac{||D^mf(0)\left(z^m\right)||}{m!}+\sum_{s=1}^{\infty}\frac{||D^{sk+m}f(0)\left(z^{sk+m}\right)||}{(sk+m)!}+\bigg(\frac{1}{||z||^m+\frac{||D^mf(0)\left(z^m\right)||}{m!}}\\&\quad+\frac{||z||^{k-m}}{1-||z||^k}\bigg)\sum_{s=1}^{\infty}\left(\frac{||D^{sk+m}f(0)\left(z^{sk+m}\right)||}{(sk+m)!}\right)^2\leq 1
\end{align*}
for $ ||z||=r\leq R_{k, m}\left(|c_0|\right) $, where $ |c_0|=\frac{||D^mf(0)\left(z_0^m\right)||}{m!} $ and  $ R_{k, m}\left(|c_0|\right) $  is the maximal positive root of the equation 
\begin{align}\label{ee-1.2}
\left(1-|c_0|-|c_0|^2\right)r^{m+k}+r^k+|c_0|r^m-1=0.
\end{align}
Each $ R_{k, m}\left(|c_0|\right) $ is sharp. 
\end{thm}
Setting $ \frac{D^mf_l(0)(z^m)}{m!}=a_lz_l^m $, $ l=1, 2, \ldots, n $, we have the following result as a consequence of Theorem \ref{BS-thm-2.7}.
\begin{cor}\label{BS-cor-2.1}
	Let $ m, k\in\mathbb{N} $, $ a=(a_1, a_2, \ldots, a_n)^{\prime} $, $ 0\leq m\leq k $, 
	\begin{align*}
	f(z)&=\left(a_1z_1^m+g_1(z), a_2z_2^m+g_2(z), \ldots, a_nz_n^m+g_n(z)\right)^{\prime}\\&=\frac{D^mf(0)(z^m)}{m!}+\sum_{s=1}^{\infty}\frac{D^{sk+m}f(0)(z^{sk+m})}{(sk+m)!}\in\mathcal{H}\left(\mathbb{D}^n, \overline{\mathbb{D}^n}\right),
	\end{align*}
$ |a_1z_1^m|=\cdots=|a_nz_n^m| $, where $ \frac{D^mf(0)(z^m)}{m!}=\left(a_1z_1^m, a_2z_2^m, \ldots, a_nz_n^m\right)^{\prime} $. Then $ \mathcal{C}^f_m(||z||)\leq 1 $ for $ ||z||=r\leq R_{k, m}\left(|c_0|\right) $, where $ R_{k, m}\left(|c_0|\right) $  is the maximal positive root of \eqref{ee-1.2}. Each $ R_{k, m}\left(|c_0|\right) $ is sharp. 
\end{cor}
\begin{cor}
We note that Corollary \ref{BS-cor-2.1} is a sharp refined version of Corollary 2.1 in \cite{Liu-Liu-JMAA-2020}.
\end{cor}
We now discuss the proof of Theorems \ref{Thm-2.5}, \ref{Thm-2.6} and \ref{BS-thm-2.7} applying Lemma \ref{Lemm-1.2}. 
\begin{proof}[\bf Proof of Theorm \ref{Thm-2.5}]
Let $z\in \mathbb{D}^n\setminus\{0\}$ be fixed and denote $z_0=z/{||z||}.$ Let us define a function $h_j(\lambda):=f_j(\lambda z_0),$ $\lambda \in \mathbb{D}.$ Then it is easy to see that $h_j\in \mathcal{H}(\mathbb{D},\overline{U})$ and 
\begin{align*}
h_j(\lambda)=a_j\left(\dfrac{z_j}{||z||}\right)^m{\lambda}^m+\sum_{s=N}^{\infty}\dfrac{D^sf_j(0)(z^s_0)}{s!}{\lambda}^s
\end{align*}
from the hypothesis of Theorem \ref{Thm-2.5}, where $j$ satisfies $|z_j|=||z||=\max_{1\leq l\leq n}\{|z_l|\}.$ We write $b_m=a_j\left(\frac{z_j}{||z||}\right)^m,$ $b_s=\dfrac{D^sf_j(0)(z^s_0)}{s!},$ $s=N,N+1,\dots.$ Then it is easy to see that the function $\omega(\lambda)=b_m+\sum_{s=N}^{\infty}b_s{\lambda}^{s-m}\in \mathcal{H}(\mathbb{D},\overline{\mathbb{D}})$ due to $h_j\in \mathcal{H}(\mathbb{D},\overline{\mathbb{D}}).$ Also from  the hypothesis, $|b_m|=|a_j|=||a||.$  In view of Lemma \ref{Lemm-1.2} for the function $h_j\in \mathcal{H}(\mathbb{D},\overline{\mathbb{D}}),$ we get that 
\begin{align*}
	\sum_{s=N}^{\infty}|b_s|{|\lambda|}^s&+sgn(t)\sum_{s=1}^{t}|b_s|^2\dfrac{{|\lambda|}^N}{1-|\lambda|}+\left(\dfrac{1}{1+|b_m|}+\dfrac{|\lambda|}{1-|\lambda|}\right)\sum_{s=t+1}^{\infty}|b_s|^2|\lambda|^{2s}\\&\leq \dfrac{(1-|b_m|^2)|\lambda|^N}{1-|\lambda|}
\end{align*} 
 and thus, we obtain the following estimate
 \begin{align*}
 	&\sum_{s=N}^{\infty}\dfrac{|D^sf_j(0)(z^s_0)|}{s!}{|\lambda|}^s+sgn(t)\sum_{s=1}^{t}\left(\dfrac{|D^sf_j(0)(z^s_0)|}{s!}\right)^2\dfrac{{|\lambda|}^N}{1-|\lambda|}\\&+\left(\dfrac{1}{1+|a_j|}+\dfrac{|\lambda|}{1-|\lambda|}\right)\sum_{s=t+1}^{\infty}\left(\dfrac{|D^sf_j(0)(z^s_0)|}{s!}\right)^2|\lambda|^{2s}\leq  \dfrac{(1-|a_j|^2)|\lambda|^N}{1-|\lambda|}\; \mbox{for}\; j=1,2,\dots,n.
 \end{align*}
Setting $ |\lambda|=||z||=r $ and $ z=z_0||z|| $, using maximum modulus principle, this implies that 
\begin{align*}
	&\sum_{s=N}^{\infty}\dfrac{||D^sf(0)(z^s)||}{s!}+sgn(t)\sum_{s=1}^{t}\left(\dfrac{||D^sf(0)(z^s)||}{s!}\right)^2\dfrac{{||z||}^{N-2s}}{1-||z||}\\&+\left(\dfrac{||z||^m}{||z||^m+\dfrac{||D^mf(0)(z^m)||}{m!}}+\dfrac{||z||}{1-||z||}\right)\sum_{s=t+1}^{\infty}\left(\dfrac{||D^sf(0)(z^s)||}{s!}\right)^2\leq  \dfrac{(1-||a||^2)||z||^N}{1-||z||}.
\end{align*}
Therefore, we obtain that

\begin{align*}
&\dfrac{||D^mf(0)(z^m)||}{m!}+\sum_{s=N}^{\infty}\dfrac{||D^sf(0)(z^s)||}{s!}+sgn(t)\sum_{s=1}^{t}\left(\dfrac{||D^sf(0)(z^s)||}{s!}\right)^2\dfrac{{||z||}^{N-2s}}{1-||z||}\\&+\left(\dfrac{||z||^m}{||z||^m+\dfrac{||D^mf(0)(z^m)||}{m!}}+\dfrac{||z||}{1-||z||}\right)\sum_{s=t+1}^{\infty}\left(\dfrac{||D^sf(0)(z^s)||}{s!}\right)^2\\&\leq ||a||r^m+ \dfrac{(1-||a||^2)r^N}{1-r}:=\mathcal{M}_{f}(||a||,r).
\end{align*}
To prove the desired inequality $ \mathcal{A}^f_{m}(||z||)\leq 1 $ for $ ||z||=r\leq r_{N, m} $, it is enough to show that the inequality $\mathcal{M}_{f}(||a||,r)\leq 1$ holds for $r\leq r_{N,m}.$ To achieve this, the analysis of the following cases is sufficient.

\noindent{\bf Case-I.} Let $m=0.$ Since $0\leq ||a||<1,$ a simple computation show that 
\begin{align*}
\mathcal{M}_{f}(||a||,r)&=||a||+(1-||a||^2)\dfrac{r^N}{1-r}\leq ||a||+2(1-||a||)\dfrac{r^N}{1-r}\\&=1+(1-||a||)\left(-1+\dfrac{2r^N}{1-r}\right)\\&\leq 1 \;\mbox{for}\, r\leq r_{N,m},
\end{align*}
where $r_{N,m}$ is root of equation $\phi_1(r)=0$ given in Lemma \ref{Lem-1.1}.

\noindent{\bf Case-II.} Let $N>2m,$ $m=1,2,\ldots,.$ Then we see that 
\begin{align*}
	\mathcal{M}_{f}(||a||,r)&\leq 1-\dfrac{r^N}{1-r}\left(||a||-\dfrac{1-r}{2r^{N-m}}\right)^2+\dfrac{4r^{2(N-m)}+4r^{N+1-2m}-4r^{N-2m}+r^2-2r+1}{4(1-r)r^{N-2m}}\\&\leq 1\;\;\mbox{for}\;\; r\leq r_{N,m},
\end{align*}
where $r_{N,m} $ is the root of equation $\phi_2(r)=0,$ given in Lemma \ref{Lem-1.1}.\vspace{1.2mm}

\noindent{\bf Case-III.} Let $m+1\leq N\leq 2m,$ $m=1,2,\ldots,.$ we deduce that 
\begin{align*}
	\mathcal{M}_{f}(||a||,r)&\leq1-\dfrac{r^N}{1-r}\left(||a||-\dfrac{1-r}{2r^{N-m}}\right)^2\\&\quad+\dfrac{4r^N+r^{2+2m-N}-2r^{1+2m-N}+r^{2m-N}+4r-4}{4(1-r)}\\&\leq 1\;\;\mbox{for}\;\; r\leq r_{N,m},
\end{align*}
where $r_{N,m}$ is the root of equation $\phi_3(r)=0$ given in Lemma \ref{Lem-1.1}.
Therefore, we conclude that $\mathcal{A}_m^f(||z||)\leq 1$ holds for $r\leq r_{N,m},$ where $r_{N,m} $ is given in Lemma \ref{Lem-1.1}.\vspace{1.2mm}

It is not difficult to check that 
$f_a(z)=(f_1(z_1), a,a,\dots,a)^\prime,$ $z=(z_1,z_2,\dots,z_n)^\prime\in U^n,$ where $f_1(z_1)=(a-z_1)/(1-az_1)$ for some $a\in [0,1]$ satisfies the condition of Theorem \ref{Thm-2.5}. Also we can write $f_1(z_1)=a-(1-a^2)\sum_{s=1}^{\infty}a^{s-1}z^s_1.$ Putting $z=(r,0,\dots,0)^\prime,$ $0\leq r<1,$ it shows that 
 $\frac{||D^sf_a(0)(z^s)||}{s!}=(1-a^2)a^{s-1}r^s$ and for $m=0,$ $\frac{||D^mf_a(0)(z^m)||}{m!}=a.$ In fact, for $f_a,$ a simple computation shows that 
\begin{align}\label{ee-1.1}
\mathcal{A}^{f_a}_{m}(r)&= a+\sum_{s=1}^{\infty}(1-a^2)a^{s-1}r^s+sgn(t)\sum_{s=1}^{t}(1-a^2)^2a^{2(s-1)}r^{2s}\dfrac{r^{N-2s}}{1-r}\\&\nonumber\quad+\left(\dfrac{1}{1+a}+\dfrac{r}{1-r}\right)\sum_{s=1}^{\infty}(1-a^2)^2a^{2(s-1)}r^{2s}\\&\nonumber=1+(1-a)G(a,r), 
\end{align}
where
$G(a,r):=-1+\frac{(1+a)r}{1-ar}+\frac{(1-a^2)r^2}{(1-r)(1-ar)}+\frac{(1+a)r^N(1-a^{2t})sgn(t)}{1-r}$
and 
\begin{align*}
	\lim\limits_{a\rightarrow 1^-}G(a,r)=-1+\dfrac{2r}{1-r}.
\end{align*}
Clearly, the right side of \eqref{ee-1.1} is greater than $1$ if $r>1/3.$ This implies that the constant $r_{1,0}=1/3$ is optimal.\vspace{1.2mm}

Next, to show that the constant $ r_{2,1}=3/5 $ is optimal, we consider the function 
$f_a(z)=(f_1(z_1), a,a,\dots,a)^\prime,$ $z=(z_1,z_2,\dots,z_n)^\prime\in \mathbb{D}^n,$ where $f_1(z_1)=z_1(a-z_1)/(1-az_1)$ for some $a\in [0,1]$ and it is easy to see that $ f_a $ satisfies the condition of Theorem \ref{Thm-2.5}. By a similar argument as above, it can be easily shown that $r_{2,1}=3/5$ is optimal. This completes the proof.   
\end{proof}
\begin{proof}[\bf Proof of Theorem \ref{Thm-2.6}]
Let  $z\in \mathbb{D}^n\setminus\{0\}$ be fixed and we denote $z_0=z/{||z||}.$ We define a function $h_l(\lambda):=f_l(\lambda z_0),$ $\lambda \in \mathbb{D},$ $l=1,2,\dots n.$ Clearly, $h_l\in \mathcal{H}(\mathbb{D},\overline{\mathbb{D}})$ and  we easily deduce that 
\begin{align*}
h_l(\lambda)=\dfrac{D^mf_l(0)\left(z_0^m\right)}{m!}{\lambda}^m+\sum_{s=N}^{\infty}\dfrac{D^{s}f_l(0)\left(z_0^{s}\right)}{s!}{\lambda}^s
\end{align*}
from the condition of Theorem \ref{Thm-2.6}.  Hence, we easily deduce that 
$\omega(\lambda)=b_m+\sum_{s=N}^{\infty}b_s{\lambda}^{s-m}\in \mathcal{H}(\mathbb{D},\overline{\mathbb{D}})$ due to $h_l\in \mathcal{H}(\mathbb{D},\overline{\mathbb{D}}),$   where $b_m=\frac{D^mf_l(0)\left(z_0^m\right)}{m!}$ and $b_s=\frac{D^{s}f_l(0)\left(z_0^{s}\right)}{s!}$ for $s=N,N+1,\ldots,.$ Because $\omega\in \mathcal{H}(\mathbb{D},\overline{\mathbb{D}}),$ in view of Lemma \ref{Lemm-1.2}, we obtain the following estimate
\begin{align*}
&\dfrac{|D^mf_l(0)\left(z_0^m\right)|}{m!}{|\lambda|}^{m}+\sum_{s=N}^{\infty}\dfrac{|D^{s}f_l(0)\left(z_0^{s}\right)|}{s!}{|\lambda|}^s+sgn(t)\sum_{s=1}^{t}\left(\dfrac{|D^{s}f_l(0)\left(z_0^{s}\right)|}{s!}\right)^2\dfrac{{|\lambda|}^N}{1-|\lambda|}\\&+\left(\dfrac{1}{1+\dfrac{|D^mf_l(0)\left(z_0^m\right)|}{m!}}+\dfrac{|\lambda|}{1-|\lambda|}\right)\sum_{s=t+1}^{\infty}\left(\dfrac{|D^{s}f_l(0)\left(z_0^{s}\right)|}{s!}\right)^2|\lambda|^{2s}\\&\leq \dfrac{|D^mf_l(0)\left(z_0^m\right)|}{m!}{|\lambda|}^{m}+\dfrac{\left(1-\left(\dfrac{|D^mf_l(0)\left(z_0^m\right)|}{m!}\right)^2\right)}{1-|\lambda|}|\lambda|^N\;\;\mbox{for}\;\; l=1,2,\dots,n.
\end{align*} 
Set $ |\lambda|=||z||=r $ and $ z=z_0||z|| $, we obtain that 
\begin{align*}
&\dfrac{||D^mf(0)\left(z^m\right)||}{m!}+\sum_{s=N}^{\infty}\dfrac{||D^sf(0)(z^s)||}{s!}+sgn(t)\sum_{s=1}^{t}\left(\dfrac{||D^sf(0)(z^s)||}{s!}\right)^2\dfrac{{||z||}^{N-2s}}{1-||z||}\\&+\left(\dfrac{||z||^m}{||z||^m+\dfrac{||D^mf(0)(z^m)||}{m!}}+\dfrac{||z||}{1-||z||}\right)\sum_{s=t+1}^{\infty}\left(\dfrac{||D^sf(0)(z^s)||}{s!}\right)^2\\&\leq ||a||||z||^m+ \dfrac{(1-||a||^2)||z||^N}{1-||z||}=||a||r^m+\dfrac{(1-||a||^2)r^N}{1-r},
\end{align*}
where $||a||=\frac{||D^mf(0)\left(z_0^m\right)||}{m!}.$ We arrive at the desired conclusions by employing a similar argument to that given in the proof of Theorem \ref{Thm-2.5}. Hence, we omit the details. With this, the theorem's proof is concluded.
\end{proof}
\begin{proof}[\bf Proof of Theorem \ref{BS-thm-2.7}]
Fix $ z\in\mathbb{C}^n\setminus\{0\} $ and set $ z_0=z/||z|| $. Letting $ h_{l}(\lambda)=f_l(\lambda z_0) $ for $ \lambda\in \mathbb{D} $, $ l=1, 2, \ldots, n $, it is easy to see that $ h_l\in\mathcal{H}(\mathbb{D}, \overline{\mathbb{D}}) $ and from the hypothesis, we can express the function $ h_l $ in the following form of a series
\begin{align*}
h_l(\lambda)&=b_m\lambda^m+\sum_{s=1}^{\infty}b_{sk+m}\lambda^{sk+m}=\frac{D^mf_l(0)\left(z_0^m\right)}{m!}\lambda^m+\sum_{s=1}^{\infty}\frac{D^{sk+m}f_l(0)\left(z_0^{sk+m}\right)}{(sk+m)!}\lambda^{sk+m}.
\end{align*}
We write $ \mu=\lambda^k $. Then it yields that 
\begin{align*}
\omega(\mu)=c_0+\sum_{s=1}^{\infty}c_s\mu^s\in\mathcal{H}(\mathbb{D}, \overline{\mathbb{D}})\; \mbox{due to}\; h_l\in \mathcal{H}(\mathbb{D}, \overline{\mathbb{D}}).
\end{align*}
Here $ c_s=b_{sk+m}=\frac{D^{sk+m}f_l(0)\left(z_0^{sk+m}\right)}{(sk+m)!},\; s=1, 2, 3, \ldots $ and $ c_0=\frac{D^mf_l(0)\left(z_0^m\right)}{m!} $. In view of Lemma \ref{Lemm-1.2} (with  $N=1$), a simple computation shows that
\begin{align*}
&|c_0|{|\lambda|}^m+\sum_{s=1}^{\infty}|c_s|{|\lambda|}^{ks+m}+\left(\dfrac{1}{1+|c_0|}+\dfrac{{|\lambda|}^k}{1-{|\lambda|}^k}\right)\sum_{s=1}^{\infty}|c_s|^2{|\lambda|}^{2ks+m}\\&\leq  {|\lambda|}^m\left(|c_0|+\dfrac{(1-|c_0|^2){|\lambda|}^k}{1-{|\lambda|}^k}\right)
\end{align*}
This gives the following estimate
\begin{align*}
&\dfrac{|D^mf_l(0)\left(z_0^m\right)|}{m!}{|\lambda|}^m+\sum_{s=1}^{\infty}\frac{|D^{sk+m}f_l(0)\left(z_0^{sk+m}\right)|}{(sk+m)!}{|\lambda|}^{ks+m}\\&\quad+\left(\dfrac{1}{|\lambda|^m+\dfrac{|D^mf_l(0)\left(z_0^m\right)|}{m!}|\lambda|^m}+\dfrac{{|\lambda|}^{k-m}}{1-{|\lambda|}^k}\right)\sum_{s=1}^{\infty}\left(\dfrac{|D^{sk+m}f_l(0)\left(z_0^{sk+m}\right)|}{(sk+m)!}\right)^2{|\lambda|}^{2(ks+m)}\\&\leq \left(|C_0|{|\lambda|}^m+\dfrac{(1-|c_0|^2){|\lambda|}^{k+m}}{1-{|\lambda|}^k}\right),\;\;\mbox{for}\;\; l=1,2,\dots,n.
\end{align*}
We set $ |\lambda|=||z||=r $ and $ z=z_0||z|| $. Then by the maximum modulus principle, a simple calculation confirms that
\begin{align}
&\nonumber\frac{||D^mf(0)\left(z^m\right)||}{m!}+\sum_{s=1}^{\infty}\frac{||D^{sk+m}f(0)\left(z^{sk+m}\right)||}{(sk+m)!}+\bigg(\frac{1}{||z||^m+\frac{||D^mf(0)\left(z^m\right)||}{m!}}\\&\nonumber\quad+\frac{||z||^{k-m}}{1-||z||^k}\bigg)\left(\sum_{s=1}^{\infty}\frac{||D^{sk+m}f(0)\left(z^{sk+m}\right)||}{(sk+m)!}\right)^2\\&\leq\nonumber ||z||^m\left(|c_0|+\left(1-|c_0|^2\right)\frac{||z||^k}{1-||z||^k}\right) \\&=\nonumber 1+\frac{\left(1-|c_0|-|c_0|^2\right){||z||}^{m+k}+{||z||}^k+|c_0|r^m-1}{1-{||z||}^k}\leq 1
\end{align}
for $ ||z||=r\leq R_{k,m}\left(|c_0|\right) $, where $ R_{k,m}\left(|c_0|\right) $ is the unique root in $ (0, 1) $ of equation \eqref{ee-1.2}. The next step is to show the constant $R_{k,m}\left(|c_0|\right)$ is sharp. To serve the purpose, we consider the function $ f_a $ defined by
\begin{align*}
f_a(z)=\left(z_1^m\frac{a-z^k_1}{1-az_1^k}, z_2^m\frac{a-z^k_2}{1-az_2^k}, \ldots, z_n^m\frac{a-z^k_n}{1-az_n^k}\right)
\end{align*}
for $ z=\left(z_1, z_2, \ldots, z_n\right)^{\prime}\in U^n $ and $ a\in [0, 1) $. In this case, we suppose that $ z=\left(z_1, 0, \ldots, 0\right)^{\prime} $ which implies that $ ||z||=|z_1|=r $, and according to the definition of the Fr$ \acute{e} $chet derivative, we get that
\begin{align*}
\begin{cases}
\dfrac{||D^{sk+m}f_a(0)\left(z^{sk+m}\right)||}{(sk+m)!}=\bigg|\dfrac{\partial^{sk+m}f_1(0)}{\partial z_1^{sk+m}}\cdot\dfrac{z_1^{sk+m}}{(sk+m)!}\bigg|\; \mbox{for}\; k\geq 0\vspace{2mm}\\
\dfrac{||D^{m}f_a(0)\left(z^{m}\right)||}{m!}=\bigg|\dfrac{\partial^{m}f_1(0)}{\partial z_1^{m}}\cdot\dfrac{z_1^{m}}{m!}\bigg|\; \mbox{for}\; s=0,
\end{cases}
\end{align*}
where
\begin{align*}
f_1(z)=z_1^m\left(\frac{a-z_1^k}{1-az_1^k}\right)=az_1^m-\left(1-a^2\right)\sum_{s=1}^{\infty}a^{s-1}z_1^{sk+m}.
\end{align*}
A simple computation gives that 
\begin{align*}
\frac{||D^mf_a(0)(z^m)||}{m!}=ar^m\;\;\mbox{and}\;\; \dfrac{||D^{sk+m}f_a(0)\left(z^{sk+m}\right)||}{(sk+m)!}=(1-a^2)a^{s-1}r^{sk+m}.
\end{align*}
Thus we see that 
\begin{align*}
&\nonumber\frac{||D^mf(0)\left(z^m\right)||}{m!}+\sum_{s=1}^{\infty}\frac{||D^{sk+m}f(0)\left(z^{sk+m}\right)||}{(sk+m)!}+\bigg(\frac{1}{||z||^m+\frac{||D^mf(0)\left(z^m\right)||}{m!}}\\&\nonumber\quad+\frac{||z||^{k-m}}{1-||z||^k}\bigg)\left(\sum_{s=1}^{\infty}\frac{||D^{sk+m}f(0)\left(z^{sk+m}\right)||}{(sk+m)!}\right)^2\\&=ar^m+\left(1-a^2\right)\sum_{s=1}^{\infty}a^{s-1}r^{sk+m}+\left(\frac{r^{-m}}{1+a}+\frac{r^{k-m}}{1-r^k}\right)\left(1-a^2\right)^2\sum_{s=1}^{\infty}a^{2s-2}r^{2sk+2m}\\&=r^m\left(a+\left(1-a^2\right)\frac{r^k}{1-r^k}\right)
\end{align*}
which is bigger than $ 1 $ if, and only if, $ r>R_{k,m}(a) $. This establishes the sharpness of $ R_{k,m}(a) $ and with this, the proof of theorem is completed.
\end{proof}
\section{Refined versions of the Bohr's inequality in complex Banach spaces}
There are a few of articles on Bohr inequality on complex Banach spaces. In this section, we shall consider refined version of Bohr's phenomenon in complex Banach spaces. Let $ X $ and $ Y $ be complex Banach spaces and $ \mathcal{B}_Y $ be the unit ball in $ Y $. For domain $ G\subset X $ and $ \Omega\subset Y $,  let $ \mathcal{H}(G, \Omega) $ be the set of all holomorphic functions from $ G $ into $ \Omega $. Any mapping $ f\in \mathcal{H}(G, \Omega) $ can be expanded in the following series 
\begin{align}\label{BS-eq-3.1}
	f(x)=\sum_{s=0}^{\infty}\frac{1}{s!}D^sf(0)\left(x^s\right),
\end{align}
where $ D^sf(0) $, $ s\in\mathbb{N} $, denote the $ s $-th Fr\'echet derivative of $ f $ at $ 0 $, which is bounded symmetric $ s $-linear mapping from $ \prod_{i=1}^{s}X $ to $ \mathbb{C} $. It is understood that $ D^0f(0)\left(x^0\right)=f(0) $. A domain $ G\subset X $ is said to be \textit{balanced}, if $ zG\subset G $ for all $ z\in \mathbb{D} $. Given a balanced domain $ G $, we denote the \emph{higher dimensional Bohr radius} by $ K^G_X(\Omega) $ the largest non-negative number $ r $ such that 
\begin{align*}
\sum_{s=1}^{\infty}\bigg|\frac{1}{s!}D^sf(0)\left(x^s\right)\bigg|\leq d(f(0), \partial\Omega)
\end{align*}
holds for all $ x\in r G $ and all holomorphic functions $ f\in\mathcal{H}(G, \Omega) $ with the expansion \eqref{BS-eq-3.1} about the origin. Here, we denote $ d $ as the \textit{Euclidean distance} between $ f(0) $ and the boundary $ \partial\Omega $ of the domain $ \Omega $. It is easy to see that the classical Bohr's inequality \eqref{BS-eq-1.1} states that $ K^{G}_{\mathbb{C}}(\mathbb{D})=1/3. $ In recent years, researchers have paid their considerable attention to the study of the Bohr inequality and its refined versions for Banach spaces. For example, Aizenberg \cite{Aizn-PAMS-2000} have obtained that for any balanced domain $ G\subset\mathbb{C}^n $, $ K^{G}_{\mathbb{C}^n}(\mathbb{D})\geq 1/3 $ and also showed that the constant $ K^{G}_{\mathbb{C}^n}(\mathbb{D})=1/3 $ is best possible in case of when $ G $ is a convex domain. Moreover, by taking a restriction on $ f\in\mathcal{H}(G, \mathbb{D}) $ such that $ f(0)=0 $ and for any balanced domain $ G\subset \mathbb{C}^n $, Liu and Ponnusamy \cite{Liu-Ponnusamy-PAMS-2021} have improved the quantity as $ K^{G}_{\mathbb{C}^n}(\mathbb{D})\geq 1/\sqrt{2} $ and obtained that the constant $ K^{G}_{\mathbb{C}^n}(\mathbb{D})= 1/\sqrt{2} $ is best possible if $ G $ is a convex domain. Furthermore, Hamada \emph{et al.} \cite{Hamada-IJM-2009} have established the generalization of the Bohr inequality to the holomorphic mappings $ f\in\mathcal{H}(G, \mathcal{B}_Y) $ for bounded balanced domain $ G $ in a Banach space $ X $ and $ \mathcal{B}_Y $ is the (homogeneous) unit ball in a complex Banach space $ Y $, and shown that the Bohr radius cannot be improved if $ \mathcal{B}_Y $ is the unit ball of a $ J^* $-algebra i.e., $ K^{G}_{X}(\mathbb{D})=1/3 $ (see \cite[Corollary 3.2]{Hamada-IJM-2009}). For a simply connected domain $ \Omega $ in the complex plane $ \mathbb{C} $, Bhowmik and Das \cite[Theorem 3]{Bhowmik-CRMath-2021} have obtained a lower bound of the quantity $ K^{G}_{X}(\mathbb{D}) $.  Also, a generalized Bohr radius $ R_{p, q}(X) $ , where $ p, q\in [1, \infty) $ is obtained in \cite{Das-CMB-2022} for complex Banach space $ X $. Moreover, a $ n- $variable version $ R^n_{p, q}(X)  $ of the quantity $ R_{p, q}(X)  $ are considered in \cite{Das-CMB-2022} and is determined $ R^n_{p, q}(X) $ for infinite dimensional complex Hilbert space $ \mathcal{H} $. Various other
results related to the multidimensional Bohr radius have appeared recently (see \cite{Aizn-PAMS-2000,Aizenberg-Djakov-PAMS-2000,Aizeberg-PLMS-2001,Arora-CVEE-2022,Lin-Liu-Ponnusamy-Acta-2023,Paulsen-PLMS-2002,Das-CMB-2022,Blasco-CM-2017,Allu-Halder-Pal-BSM-2023} and references therein).\vspace{1.2mm}

Motivated by the work of Ali \emph{et al.} \cite{Ali-Bar-Sol-JMAA-2017} and \cite[Corollary 1]{Kayumov-Pon-CMFT-2017}, a problem concerning symmetric analytic functions was raised (see \cite[Problem 1]{Kayumov-Ponnusamy-JMAA-2018}) which is answered completely by establishing the following result. 
\begin{thm}\cite{Kayumov-Ponnusamy-JMAA-2018}\label{BS-thm-3.1}
Given $ k, m\in\mathbb{N} $, $ f(z)=\sum_{s=0}^{\infty}a_{sk+m}z^{sk+m}\in \mathcal{H}(\mathbb{D}, \mathbb{D}) $. Then 
\begin{align*}
\sum_{s=0}^{\infty}|a_{sk+m}z^{sk+m}|\leq 1\; \mbox{for}\; r\leq r_{k,m},
\end{align*}
where $ r_{k,m} $ is the maximal positive root of the equation \eqref{BS-eqqq-2.1}. The number $ r_{k,m} $ is the best possible.
\end{thm}
A multidimensional generalization of Theorem \ref{BS-thm-3.1} is established recently in \cite{Arora-CVEE-2022} for functions with lacunary series in the class $ \mathcal{H}(G, \mathbb{D}) $. The result is
\begin{thm}\cite{Arora-CVEE-2022}\label{BS-thm-3.2}
Let $ k, m\in\mathbb{N} $, $ 0\leq m\leq k $. Suppose that $ G $ be a bounded balanced domain in a complex Banach space $ X $ and $ f\in\mathcal{H}(G, \mathbb{D}) $ be of the form 
$f(x)=\frac{D^mf(0)\left(x^m\right)}{m!}+\sum_{s=1}^{\infty}\frac{D^{sk+m}f(0)\left(x^{sk+m}\right)}{(sk+m)!}$.
Then we have
\begin{align*}
\frac{|D^mf(0)\left(x^m\right)|}{m!}+\sum_{s=1}^{\infty}\frac{|D^{sk+m}f(0)\left(x^{sk+m}\right)|}{(sk+m)!}\leq 1\;\mbox{for}\; x\in \left(r_{k,m}\right)G.
\end{align*} 
Here the constant $ r_{k,m} $ is the maximal positive root in $ (0, 1) $ of the equation \eqref{BS-eqqq-2.1}. The radius $ r_{k,m} $ is best possible.
\end{thm}
For further improvement of the inequality in Theorem \ref{BS-thm-3.2}, it is natural to raise the following question.
\begin{ques}\label{BS-qn-3.1}
Can we establish an analogue of Theorem \ref{Thmm-2.1} which is a sharp refined version of Theorem \ref{BS-thm-3.2}?
\end{ques}
Ponnusamy \emph{et al.}\cite{Ponnusamy-Vijayak-Wirths-RM-2020} established a refined version of the Bohr's inequality in the case $f\in \mathcal{H}(\mathbb{D}, \mathbb{D}) $ with $f(0)=0$. We recall their result.
\begin{thm}\label{Thmm-3.1}
Suppose that $f(z)=\sum_{n=1}^{\infty}a_nz^n\in \mathcal{H}(\mathbb{D}, \mathbb{D}).$ Then
\begin{align*}
\sum_{n=1}^{\infty}|a_n|r^n+\left(\dfrac{1}{1+|a_1|}+\dfrac{r}{1-r}\right)\sum_{n=2}^{\infty}|a_n|^2r^{2n-1}\leq 1\;\;\mbox{for}\;\; r\leq\dfrac{3}{5}.
\end{align*}
The number $3/5$ is sharp.
\end{thm}
We now state our result answering Question \ref{BS-qn-3.1} for functions with lacunary series in the class $ \mathcal{H}(G, \Omega) $. In order to obtain the sharp estimate, we use a recent approach of Liu \emph{et al.} \cite{Liu-Liu-Ponnusamy-2021} which they used to investigate the Bohr radius for symmetric function $ f\in\mathcal{H}(\mathbb{D}, \mathbb{D}) $. In the last section, we show that an application of Theorem \ref{BS-thm-6.1} will help us to establish a multidimensional refined version of the Bohr-type inequality for the harmonic function $f$ from the bounded balanced domain $G$ into $\mathbb{D}$.
\begin{thm}\label{BS-thm-6.1}
Let $ k, m\in\mathbb{N} $, $ 0\leq m\leq k $. Suppose that $ G $ be a bounded balanced domain in a complex Banach space $ X $ and $ f\in\mathcal{H}(G, \mathbb{D}) $ be of the form 
$f(x)=\frac{D^mf(0)\left(x^m\right)}{m!}+\sum_{s=1}^{\infty}\frac{D^{sk+m}f(0)\left(x^{sk+m}\right)}{(sk+m)!}$.
Then 
\begin{align*}
\mathcal{I}^f_{m,k}(r):=&\frac{|D^mf(0)\left(x^m\right)|}{m!}+\sum_{s=1}^{\infty}\frac{|D^{sk+m}f(0)\left(x^{sk+m}\right)|}{(sk+m)!}\\&\quad+\left(\frac{1}{r^m+\frac{|D^mf(0)\left(x^m\right)|}{m!}}+\frac{r^{k-m}}{1-r^k}\right)\sum_{s=1}^{\infty}\left(\frac{|D^{sk+m}f(0)\left(x^{sk+m}\right)|}{(sk+m)!}\right)^2\leq 1
\end{align*}
for $ x\in \left(R_{k,m}(|c_0|)\right)G $, where $ R_{k,m}(|c_0|) $ is the maximal positive root in $ (0, 1) $ of the equation given by \eqref{ee-1.2}.
\end{thm}
As a consequence of Theorem \ref{BS-thm-6.1} (for $ m=0 $ and $ k=1 $), we obtain the following corollary which is in fact a refined version of \cite[Corollary 3.2]{Hamada-IJM-2009}.
\begin{cor}
Let $ f\in\mathcal{H}(G, \mathbb{D}) $ be of the form $ f(x)=\sum_{s=0}^{\infty}\frac{D^sf(0)(x^s)}{s!} $, where $ D^sf(0) $, $ s\in\mathbb{N} $ denote the $ s $-th Fr\'echet derivative of $ f $ at $ 0 $. Then 
\begin{align*}
|f(0)|+\sum_{s=1}^{\infty}\frac{|D^sf(0)\left(x^s\right)|}{s!}+\left(\frac{1}{1+|f(0)|}+\frac{r}{1-r}\right)\sum_{s=1}^{\infty}\left(\frac{|D^sf(0)\left(x^s\right)|}{s!}\right)^2\leq 1
\end{align*}
for $ x\in (1/3)G $. Here, the constant $ 1/3 $ is best possible.
\end{cor}
Now we concentrate to obtain an analogue of Theorem \ref{Thmm-3.1} for a bounded balanced domain $ G $ in a complex Banach space $ X $ and $ f\in\mathcal{H}(G, \mathbb{D}) $ with lacunary series and we obtain the following sharp refined version of Bohr's inequality. 
\begin{thm}\label{BS-thm-6.4}
	Let $ G $ be a bounded balanced domain in a complex Banach space $ X $ and $ f\in\mathcal{H}(G, \mathbb{D}) $ be of the form $ f(x)=\sum_{s=1}^{\infty}\frac{D^sf(0)(x^s)}{s!} $. Then we have 
\begin{align}\label{ee-3.2}
\sum_{s=1}^{\infty}\frac{|D^sf(0)(x^s)|}{s!}+\left(\frac{1}{r+\frac{|Df(0)(x)|}{1!}}+\frac{1}{1-r}\right)\sum_{s=1}^{\infty}\left(\frac{|D^sf(0)(x^s)|}{s!}\right)^2\leq 1
\end{align}
for $ x\in(3/5)G $. The constant $ 3/5 $ is best possible.
\end{thm}

\begin{proof}[\bf Proof of Theorem \ref{BS-thm-6.1}]
	Assume any fixed $ y\in G $ and let $ F(z):=f(zy) $, $ z\in \mathbb{D} $. Then it is easy to see that $ F : \mathbb{D}\rightarrow \mathbb{D} $ is holomorphic and 
	\begin{align*}
		F(z)=\frac{D^mf(0)\left(y^m\right)}{m!}z^m+\sum_{s=1}^{\infty}\frac{D^{sk+m}f(0)\left(y^{sk+m}\right)}{(sk+m)!}z^{sk+m}=z^mg(z^k),
	\end{align*}
	where $ g(z):=c_0+\sum_{s=1}^{\infty}c_sz^{s}\in \mathcal{H}\left(U, U\right) $ and 
	\begin{align*}
		c_0=\frac{D^mf(0)\left(y^m\right)}{m!}\; \mbox{and}\; c_s=\frac{D^{sk+m}f(0)\left(y^{sk+m}\right)}{(sk+m)!},\; s=1, 2, \ldots
	\end{align*}
	Because $g\in \mathcal{H}\left(U, U\right),$ in view of Lemma \ref{Lemm-1.2} (with $N=1$), we have
	\begin{align*}
		\sum_{s=0}^{\infty}|c_s||z|^{zk}+\left(\frac{1}{1+|c_0|}+\frac{|z|^k}{1-|z|^k}\right)\sum_{s=1}^{\infty}|c_s|^2|z|^{2sk}\leq |c_0|+\left(1-|c_0|^2\right)\frac{|z|^k}{1-|z|^k}
	\end{align*}
	and multiplying both sides by $ |z|^m $ and puting the value of $c_s$, we obtain
\begin{align*}
&\sum_{s=0}^{\infty}\frac{|D^{sk+m}f(0)\left(y^{sk+m}\right)|}{(sk+m)!}|z|^{sk+m}\\&\quad+\left(\frac{1}{|z|^m+\frac{|D^mf(0)\left(y^m\right)|}{m!}|z|^m}+\frac{|z|^{k-m}}{1-|z|^k}\right)\sum_{s=1}^{\infty}\left(\frac{|D^{sk+m}f(0)\left(y^{sk+m}\right)|}{(sk+m)!}\right)^2|z|^{2sk+2m}\\&\leq |z|^m|c_0|+\left(1-|c_0|^2\right)\frac{|z|^{k+m}}{1-|z|^k}\\&=1+\frac{\left(\left(1-|c_0|-|c_0|^2\right)\right)|z|^{m+k}+|z|^k+|c_0||z|^m-1}{1-|z|^k}.
\end{align*}
Therefore, for the setting $ |z|=r $ and $ x=y|z| $, the inequality  $\mathcal{I}^f_{m,k}(r)\leq 1$ holds 
for $ x\in \left(R_{k,m}(|c_0|)\right)G $, where $ R_{k,m}(|c_0|) $ is the maximal positive root in $ (0, 1) $ of \eqref{ee-1.2}. \vspace{1.2mm}

To prove the constant $ R_{k,m}(|c_0|) $ cannot be improved, we use a technique similar to that in the proof of \cite{Arora-CVEE-2022} and \cite{Hamada-IJM-2009}. We prove that the inequality $\mathcal{I}^f_{m,k}(r)\leq 1$ is not holds for $ x\in r_0G $, where $ r_0\in \left(R_{k,m}(|c_0|), 1\right) $. As we know that there exists a $ c\in (0, 1) $ and $ \gamma\in\partial G $ such that $ cr_0>R_{k,m}(|c_0|) $ and $ c\sup\{||x|| : x\in \partial G\}<||\gamma|| $. Let us  consider a function $ h $ on $ G $ defined by 
\begin{align*}
h(x):=W\left(\frac{c\Psi_{\gamma}(x)}{||\gamma||}\right)\;\mbox{and}\; W(z):=z^m\left(\frac{a-z^k}{1-az^k}\right),
\end{align*}
where $ \Psi_{\gamma} $ is a bounded linear functional on $ X $ with $ \Psi_{\gamma}(\gamma)=||\gamma|| $, $ ||\Psi_{\gamma}||=1 $, and $ a\in [0, 1) $. It is easy to check that $ c\Psi_{\gamma}(x)/||\gamma||\in \mathbb{D} $ and $ h\in\mathcal{H}(G, \mathbb{D}) $. Choosing $ x=r_0\gamma $, we get 
\begin{align*}
h(r_0\gamma)=\left(cr_0\right)^m\left(\frac{a-\left(cr_0\right)^k}{1-a\left(cr_0\right)^k}\right)=\left(cr_0\right)^m\left(a-\left(1-a^2\right)\sum_{s=1}^{\infty}a^{s-1}\left(cr_0\right)^{sk}\right).
\end{align*}
Thus, a tedious computation gives that
\begin{align*}
&\sum_{s=0}^{\infty}\frac{|D^{sk+m}f(0)\left(y^{sk+m}\right)|}{(sk+m)!}\left(cr_0\right)^{sk+m}\\&\quad+\left(\frac{1}{(cr_0)^m+\frac{|D^mf(0)\left(y^m\right)|}{m!}(cr_0)^m}+\frac{(cr_0)^{k-m}}{1-(cr_0)^k}\right)\sum_{s=1}^{\infty}\left(\frac{|D^{sk+m}f(0)\left(y^{sk+m}\right)|}{(sk+m)!}\right)^2\left(cr_0\right)^{2sk+2m}\\&=\sum_{s=0}^{\infty}\frac{|D^{sk+m}f(0)\left(y^{sk+m}\right)|}{(sk+m)!}\left(cr_0\right)^{sk+m}\\&\quad+\left(\frac{1}{1+\frac{|D^mf(0)\left(y^m\right)|}{m!}}+\frac{(cr_0)^{k}}{1-(cr_0)^k}\right)\sum_{s=1}^{\infty}\left(\frac{|D^{sk+m}f(0)\left(y^{sk+m}\right)|}{(sk+m)!}\right)^2\left(cr_0\right)^{2sk+m}\\&=(cr_0)^ma+\left(1-a^2\right)\sum_{s=1}^{\infty}a^{s-1}\left(cr_0\right)^{sk+m}+\left(\frac{1}{1+a}+\frac{\left(cr_0\right)^k}{1-(cr_0)^k}\right)\left(1-a^2\right)^2\sum_{s=1}^{\infty}a^{2s-2}\left(cr_0\right)^{2ks+m}\\&=\left(cr_0\right)^m\left(a+\left(1-a^2\right)\frac{(cr_0)^k}{1-(cr_0)^k}\right)\\&=1+\frac{\left(\left(1-a-a^2\right)(cr_0)^{m+k}+\left(cr_0\right)^k+a\left(cr_0\right)^m-1\right)}{1-\left(cr_0\right)k}>1.
\end{align*}
 This shows that the constant $ R_{k,m}(|c_0|) $ cannot be improved.
\end{proof}
\begin{proof}[\bf Proof of Theorem \ref{BS-thm-6.4}]
Assume any fixed $ y\in G $ and let $ F(z):=f(zy) $ for $ z\in \mathbb{D} $. Then it is easy to see that $ F : \mathbb{D}\rightarrow \mathbb{D} $ is holomorphic and 
\begin{align*}
F(z)=\sum_{s=1}^{\infty}\frac{D^sf(0)(y^s)}{s!}z^s=\sum_{s=1}^{\infty}b_sz^s=:\varphi(z),
\end{align*}
where $ b_s=\frac{D^sf(0)(y^s)}{s!} $ for s=1, 2, $\ldots$ and $ \varphi(z)=z\sum_{s=1}^{\infty}b_sz^{s-1}=z\sum_{s=0}^{\infty}B_sz^s $, where $ B_s:=b_{s+1} $ for $ s=0, 1, 2, \ldots$. Clearly, $ \sum_{s=0}^{\infty}B_sz^s\in\mathcal{H}(\mathbb{D}, \overline{\mathbb{D}}) $. In view of Lemma \ref{Lemm-1.2} (with $ N=1 $), we must have
\begin{align*}
\sum_{s=0}^{\infty}|B_s||z|^s+\left(\frac{1}{1+|B_0|}+\frac{|z|}{1-|z|}\right)\sum_{s=1}^{\infty}|B_s|^2|z|^{2s}\leq |B_0|+\left(1-|B_0|^2\right)\frac{|z|}{1-|z|}
\end{align*}
which implies that
\begin{align*}
\sum_{s=0}^{\infty}|b_{s+1}||z|^s+\left(\frac{1}{1+|b_1|}+\frac{|z|}{1-|z|}\right)\sum_{s=1}^{\infty}|b_{s+1}|^2|z|^{2s}\leq |B_0|+\left(1-|B_0|^2\right)\frac{|z|}{1-|z|}.
\end{align*}
In fact, we have 
\begin{align*}
&\sum_{s=1}^{\infty}\frac{|D^{s+1}f(0)(y^{s+1})|}{(s+1)!}|z|^s+\left(\frac{1}{1+\frac{|Df(0)(y)|}{1!}}+\frac{|z|}{1-|z|}\right)\sum_{s=1}^{\infty}\left(\frac{|D^{s+1}f(0)(y^{s+1})|}{(s+1)!}\right)^{2}|z|^{2s}\\&\leq |B_0|+\left(1-|B_0|^2\right)\frac{|z|}{1-|z|}.
\end{align*}
Multiplying both sides by $ |z| $, the above inequality takes the following form
\begin{align*}
&\sum_{s=1}^{\infty}\frac{|D^{s+1}f(0)(y^{s+1})|}{(s+1)!}|z|^{s+1}+\left(\frac{|z|^2}{|z|+\frac{|Df(0)(y)|}{1!}|z|}+\frac{|z|^2}{1-|z|}\right)\sum_{s=1}^{\infty}\left(\frac{|D^{s+1}f(0)(y^{s+1})|}{(s+1)!}\right)^{2}|z|^{2s+1}\\&\leq |B_0||z|+\left(1-|B_0|^2\right)\frac{|z|^2}{1-|z|}.
\end{align*}
Setting $ |z|=r $ and $ x=y|z| $, we easily obtain 
\begin{align*}
&\sum_{s=1}^{\infty}\frac{|D^{s+1}f(0)(x^{s+1})|}{(s+1)!}+\left(\frac{1}{r+\frac{|Df(0)(x)|}{1!}}+\frac{1}{1-r}\right)\sum_{s=1}^{\infty}\left(\frac{|D^{s+1}f(0)(x^{s+1})|}{(s+1)!}\right)^{2}\\&\leq |B_0|r+\left(1-|B_0|^2\right)\frac{r^2}{1-r}\\&\leq 1+\frac{J(|B_0|)}{1-r},
\end{align*}
where $ J(t):=-1+r+tr(1-r)+\left(1-t^2\right)r^2 $. Thus the desired inequality holds if $ J(t)\leq 0 $ for $ t\leq 3/5 $. It can be easily shown that $J(t)$ has a critical point at $ t_0=(1-r)/2r $  and $ J(t) $ has maximum at $ t_0 $. This amounts of observations leads us to get
\begin{align*}
J(t)\leq J(t_0)=\frac{(5r-3)(r+1)}{4}\leq 0\; \mbox{for}\; r\leq \frac{3}{5}.
\end{align*}
Therefore, the desired inequality is established. \vspace{2mm}

Finally, we show that inequality \eqref{ee-3.2} is not hold for $ x\in r_0G $, where $ r_0\in (3/5, 1) $. As we know that there exists $ c\in (0, 1) $ and $ \gamma\in \partial G $ such that $ cr_0>3/5 $ and $ c\sup\{||x|| : x\in \partial G\}\leq ||\gamma|| $. Now, we consider a function $ h $ on $ G $ defined by 
\begin{align*}
h(x):=\omega\left(\frac{c\Psi_{\gamma}(x)}{||\gamma||}\right)\; \mbox{and}\; \omega(z):=z\left(\frac{a-z}{1-az}\right),
\end{align*}
where $ \Psi_{\gamma} $ is a bounded linear functional on $ X $ with $\Psi_{\gamma}(\gamma)=||\gamma||$, $ ||\Psi_{\gamma}||=1 $, and $ a\in [0, 1) $. It is easy to check that $ c\Psi_{\gamma}(x)/||\gamma||\in \mathbb{D} $ and $ h\in\mathcal{H}(G, \mathbb{D}) $. Choose $ x=r_0\gamma $, we get 
\begin{align*}
h(r_0\gamma)=\left(cr_0\right)\left(\frac{a-(cr_0)}{1-a(cr_0)}\right)=a(cr_0)-\left(1-a^2\right)\sum_{s=1}^{\infty}a^{s-1}(cr_0)^{s+1}.
\end{align*}
By a routine computation, we get that
\begin{align*}
&|D^sh(0)(y)|(cr_0)+\sum_{s=1}^{\infty}\frac{|D^sh(0)(y^s)|}{s!}(cs_0)^s\\&\quad+\left(\frac{1}{cr_0+a(cr_0)}+\frac{1}{1-cr_0}\right)\sum_{s=1}^{\infty}\left(\frac{|D^sh(0)(y^s)|}{s!}(cs_0)^s\right)^2\\&=a(cr_0)+\left(1-a^2\right)\sum_{s=2}^{\infty}a^{s-2}(cr_0)^s+\left(\frac{1}{cr_0+a(cr_0)}+\frac{1}{1-(cr_0)}\right)\sum_{s=2}^{\infty}\left(1-a^2\right)^2a^{2s-4}(cr_0)^{2s}\\&=a(cr_0)+\left(1-a^2\right)\frac{(cr_0)^2}{1-a(cr_0)}\\&>1.
\end{align*}
This shows that the constant $ 3/5 $ is sharp. 
\end{proof}
\subsection{Refined Bohr inequality for harmonic functions in balanced domains} 
Methods of harmonic mappings have been applied to study and solve the fluid flow problems (see \cite{Aleman-2012,Constantin-2017}). For example, in 2012, Aleman and Constantin \cite{Aleman-2012} established a connection between harmonic mappings and ideal fluid flows. In fact, Aleman and Constantin have developed an ingenious technique to solve the incompressible two dimensional Euler equations in terms of univalent harmonic mappings (see \cite{Constantin-2017} for details).\vspace{1.2mm}

A complex-valued function $ f(z)=u(x,y)+iv(x,y) $ is called harmonic in $ U $ if both $ u $ and $ v $ satisfy the Laplace's equation $ \bigtriangledown^2 u=0 $ and $ \bigtriangledown^2 v=0 $, where
\begin{equation*}
	\bigtriangledown^2:=\frac{\partial^2}{\partial x^2}+\frac{\partial^2}{\partial y^2}.
\end{equation*}
\par It is well-known that under the assumption $ g(0)=0 $, the harmonic function $ f $ has the unique canonical representation $ f=h+\overline{g} $, where $ h $ and $ g $ are analytic functions in $ U $, respectively called, analytic and co-analytic parts of $ f $.  If in addition $ f $ is univalent then we say that $ f $ is univalent harmonic on a domain $ \Omega $. A locally univalent harmonic mapping $ f=h+\overline{g} $ is sense-preserving whenever its Jacobian $ J_f(z):=|f_{z}(z)|^2-|f_{\bar{z}}(z)|=|h^{\prime}(z)|^2-|g^{\prime}(z)|^2>0 $ for $ z\in U $.\vspace{1.2mm} 

 In $ 2010 $, Abu-Muhanna \cite{Abu-CVEE-2010} considered first time the Bohr radius for the class of complex-valued harmonic function $ f=h+\bar{g} $ defined in $U$ with $|f(z)|<1$ for all $z\in U.$ After this, Kayumov \emph{et al.}\cite{Kayumov-Ponnusamy-Shakirov-MN-2017} studied Bohr radius for locally univalent harmonic mappings, Liu and Ponnusamy \cite{Liu-Ponnusamy-BMMSS-2019} have determined the Bohr radius for $k$-quasiconformal  harmonic mappings, Evdoridis \emph{et al.}\cite{Evd-Pon-Ras-Antti-IM-2019} studied an improved version of  the Bohr's inequality for locally univalent harmonic mappings, Ahamed \cite{Aha-CMFT-2022,MBA-CVEE-2022} have studied refined Bohr-Rogosinski inequalities for certain classes of harmonic mappings. recently, Arora \cite{Arora-CVEE-2022} have studied Bohr-type inequality for harmonic functions with lacunary series in complex Banach space. \vspace{1.2mm}  

 A harmonic mapping $ f=h+\bar{g} $ defined in a bounded balanced domain  $ G $ into $ \mathbb{D} $ can be expressed in lacunary series as 
 \begin{align}\label{BS-eq-3.3}
 	f(x)=\sum_{s=0}^{\infty}\frac{D^{sk+m}h(0)\left(x^{sk+m}\right)}{(sk+m)!}+\overline{\sum_{s=0}^{\infty}\frac{D^{sk+m}g(0)\left(x^{sk+m}\right)}{(sk+m)!}},
 \end{align}
where $ h $ and $ g $ are in $ \mathcal{H}(G, \mathbb{D}) $. It is established in \cite{Arora-CVEE-2022} that 
\begin{align}\label{BS-eq-3.4}
	\sum_{s=0}^{\infty}\left(\frac{|D^{sk+m}h(0)\left(x^{sk+m}\right)|}{(sk+m)!}+\frac{|D^{sk+m}g(0)\left(x^{sk+m}\right)|}{(sk+m)!}\right)\leq 2
\end{align}
for $x\in(r_{k,m})G,$ where $r_{k,m} $ is the maximal positive root of equation \eqref{BS-eqqq-2.1} and the radius $r_{k,m} $ is best possible.\vspace{1.1mm}

In view of the above observations, it is natural to ask that \textit{can we obtain a multidimensional refined version of the Bohr-type inequality for the harmonic function $f$ from the bounded balanced domain $G$ into $\mathbb{D}$?} We have given an affirmative answer to this question considering a refined version of \eqref{BS-eq-3.4} by establishing the following sharp result. 
\begin{thm}\label{BS-Thm-6.2}
Let $ G $ be bounded domain in a complex Banach space $ X $. Suppose that $ k, m\in\mathbb{N} $ with $ 0\leq m\leq k $, and $ f=h+\bar{g} $ is harmonic in $ G $ into $ \mathbb{D} $ be given by \eqref{BS-eq-3.3}. Then the inequality $ A_h(x)+A_g(x)\leq 2 $ holds for $ x\in (R_{m,k}(|c_0|))G $, where $ R_{m,k}(|c_0|) $ is the maximal root in $ (0, 1) $ of equation \eqref{ee-1.2}, where we define 
\begin{align*}
A_V(x):=&\displaystyle\sum_{s=0}^{\infty}\bigg|\frac{D^{sk+m}V(0)\left(x^{sk+m}\right)}{(sk+m)!}\bigg|\\&\displaystyle\quad+\left(\frac{1}{r^m+\frac{|D^mV(0)\left(x^m\right)|}{m!}}+\frac{r^{k-m}}{1-r^k}\right)\sum_{s=1}^{\infty}\left(\bigg|\frac{D^{sk+m}V(0)\left(x^{sk+m}\right)}{(sk+m)!}\bigg|\right)^2
\end{align*}
for $V=h,g$. The radius $ R_{m,k}(|c_0|) $ is best possible.
\end{thm}

\begin{proof}[\bf Proof of Theorem \ref{BS-Thm-6.2}]
By assumption $ f=h+\bar{g} $ and $ h, g\in\mathcal{H}(G, \mathbb{D}) $. Therefore, by applying Theorem \ref{BS-thm-6.1} to the functions 
\begin{align*}
h(x)=\sum_{s=0}^{\infty}\frac{D^{sk+m}h(0)\left(x^{sk+m}\right)}{(sk+m)!}\; \mbox{and}\; g(x)=\sum_{s=0}^{\infty}\frac{D^{sk+m}g(0)\left(x^{sk+m}\right)}{(sk+m)!},
\end{align*}
we easily obtain that $ A_h(x)\leq 1 $ and $ A_g(x)\leq 1 $ for $ x\in (R_{m,k}(|c_0|))G $, where $ R_{m,k}(|c_0|) $ is the maximal root in $ (0, 1) $ of equation \eqref{ee-1.2}. Adding these two inequalities yields that 
\begin{align*}
A_h(x)+A_g(x)\leq 2\; \mbox{for}\; x\in (R_{m,k}(|c_0|))G.
\end{align*} 
Clearly, the desired inequality is established. Next, in order to show the constant $ R_{m,k}(|c_0|) $ is best  possible, we use similar argument that is being used in the proof of Theorem \ref{BS-thm-6.1}. Henceforth, we consider the function $ h(r_0\gamma) $ given in above example with the same choice of $ r_0 $, $ c $ and $ \gamma $. Therefore, for $ |\lambda|=1 $, it is easy to see that 
\begin{align*}
h(r_0\gamma)+\overline{\lambda h(r_0\gamma)}=(1+\bar{\lambda})\left(cr_0\right)^m\left(a-\left(1-a^2\right)\sum_{s=0}^{\infty}a^{s-1}\left(cr_0\right)^{sk}\right).
\end{align*}
Hence, an easy computation shows that 
\begin{align*}
A_h(x)+A_g(x)=&2\left(cr_0\right)^m\bigg(a+\left(1-a^2\right)\sum_{s=0}^{\infty}a^{s-1}\left(cr_0\right)^{sk}\\&\quad+\left(\frac{1}{1+a}+\frac{(cr_0)^k}{1-(cr_0)^k}\right)\frac{(1-a^2)^2(cr_0)^2}{1-a^2(cr_0)^{2k}}\bigg)>2
\end{align*}
and with this shows that the constant $ R_{m,k}(|c_0|) $ is sharp. In fact, the argument used to establish the best possible part in the proof of Theorem \ref{BS-thm-6.1} will also show the last inequality on the right side.
\end{proof}

\noindent{\bf Acknowledgment:} The authors would like to thank the anonymous referees for their elaborate comments and suggestions which will improve significantly the presentation of the paper. 
\vspace{2.8mm}

\noindent\textbf{Compliance of Ethical Standards}\\

\noindent\textbf{Conflict of interest} The authors declare that there is no conflict  of interest regarding the publication of this paper.\vspace{1.5mm}

\noindent\textbf{Data availability statement}  Data sharing not applicable to this article as no datasets were generated or analysed during the current study.

\end{document}